\documentclass[12pt]{amsart}
\usepackage{amssymb,latexsym}
\usepackage{enumerate}

\makeatletter
\@namedef{subjclassname@2010}{%
  \textup{2010} Mathematics Subject Classification}
\makeatother

\newtheorem{thm}{Theorem}[section]
\newtheorem{cor}[thm]{Corollary}
\newtheorem{lem}[thm]{Lemma}

\newtheorem{prop}[thm]{Proposition}

\theoremstyle{definition}
\newtheorem{defin}[thm]{Definition}
\newtheorem{rem}[thm]{Remark}
\newtheorem{exa}[thm]{Example}

\numberwithin{equation}{section}

\DeclareMathOperator{\Char}{char}
\DeclareMathOperator{\Gal}{Gal}

\newcommand{\inv}{^{-1}}
\newcommand{\isom}{\cong}
\newcommand{\sep}{\mathrm{sep}}
\newcommand{\sig}{\sigma}
\newcommand{\dbC}{\mathbb{C}}
\newcommand{\dbF}{\mathbb{F}}
\newcommand{\dbZ}{\mathbb{Z}}
\newcommand{\grm}{\mathfrak{m}}
\newcommand{\calL}{\mathcal{L}}

\begin{document}


\baselineskip=17pt



\title[Small Galois groups]{Small Galois groups that encode valuations}

\author[I. Efrat]{Ido Efrat}
\address{Mathematics Department\\
Ben-Gurion University of the Negev\\
P.O.\ Box 653, Be'er-Sheva 84105\\
Israel}
\email{efrat@math.bgu.ac.il}

\author[J. Min\'a\v c]{J\'an Min\'a\v c}
\address{Mathematics Department\\
University of Western Ontario\\
London\\ Ontario\\ Canada N6A 5B7}
\email{minac@uwo.ca}

\date{}

 
\begin{abstract}
Let $p$ be a prime number and let $F$ be a field containing a root of unity of order $p$.
We prove that a certain relatively small canonical Galois group
$(G_F)_{[3]}$ over $F$ encodes the valuations on $F$ whose
value group is not $p$-divisible and which satisfy a variant of Hensel's lemma.
\end{abstract}




\maketitle

\section{Introduction}
A repeated phenomena in Galois theory is that essential arithmetical information on a field is encoded
in the group-theoretic structure of its canonical Galois groups.
A prototype of this phenomena is the classical Artin--Schreier theorem: a field $F$ has an ordering if and only if
its absolute Galois group $G_F=\Gal(F_\sep/F)$ contains a (non-trivial) involution.
As shown by Becker \cite{Becker74}, the same holds when $G=G_F$ is replaced by its maximal pro-$2$ quotient $G(2)$.
Moreover, the second author and Spira \cite[Th.\ 2.7]{MinacSpira90} established
a similar correspondence for an even smaller
pro-$2$ Galois group of $F$, the \textbf{$W$-group} of $F$.

In this paper we consider a generalization of the $W$-group to the pro-$p$ context,
and prove an analogous result for valuations.
Here $p$ is an arbitrary fixed prime number, and we assume that $F$ contains a root of unity of order $p$
(in particular, $\Char\,F\neq p$).
We set $G^{(2)}=G^p[G,G]$ and $G_{(3)}=G^{\delta p}[G^{(2)},G]$, where $\delta=1$ if $p>2$, and $\delta=2$
if $p=2$.
The pro-$p$ Galois group we consider is $G_{[3]}=G/G_{(3)}$.
It has exponent $\delta p$, and when $p=2$ it coincides with the $W$-group of $F$ \cite[Remark 2.1(1)]{EfratMinac2}.

Of course, a field always carries the trivial valuation, so one is only interested in valuations $v$ satisfying certain
natural requirements.
In the pro-$p$ context, such requirements on $v$ are:
\begin{enumerate}
\item[(i)]
$v(F^\times)\neq pv(F^\times)$;
\item[(ii)]
\textbf{$(F^\times)^p$-compatibility}: $1+\grm_v\leq (F^\times)^p$, where $1+\grm_v$ is the group of
$1$-units of $v$, i.e., all elements $x$ of $F$ with $v(x-1)>0$.
\end{enumerate}
Thus (i) is a strong form of non-triviality, whereas (ii) is a variant of Hensel's lemma.
Indeed, when the residue field $\bar F_v$ has characteristic not $p$, (ii) is equivalent to the validity of
Hensel's lemma relative to the maximal pro-$p$ extension $F(p)$ (\cite[Prop.\ 1.2]{Wadsworth83},
\cite[Prop.\ 18.2.4]{EfratBook}).

Our main result (Corollary \ref{valuations and the dual}) is that, under a finiteness assumption
and the hypothesis that $-1$ is a square if $p=2$,
there exists a valuation $v$ on $F$ satisfying (i) and (ii) above
if and only if the center $Z(G_{[3]})$ has a nontrivial image in $G^{[2]}=G/G^{(2)}$.

Note that when  $\Char\,\bar F_v\neq p$, conditions (i) and (ii) give a description
of the full maximal pro-$p$ Galois group $G_F(p)=\Gal(F(p)/F)$ of $F$ as a semi-direct product
$\dbZ_p^m\rtimes G_{\bar F_v}(p)$ where $m=\dim_{\dbF_p}(v(F^\times)/pv(F^\times))$
and the action is given by the cyclotomic character \cite[Example 22.1.6]{EfratBook}.

The proof of the main theorem is based on two key ingredients.
First, results of Villegas, Spira and the authors (see Theorems \ref{intersection for p odd}
and \ref{intersection for p=2} below) give an explicit list $\calL_p$ of small finite $p$-groups such that,
for $G=G_F$ as above,
\[
G_{(3)}=\bigcap\{N\trianglelefteq G\ | \ G/N\in\calL_p\}.
\]

A second ingredient is the notion of \textbf{$p$-rigid} elements in $F$
(see \S\ref{section on rigidity} for the definition).
In a series of works by Arason, Elman, Hwang, Jacob, Ware, and the first author
(see \cite{Jacob81},  \cite{Ware81}, \cite{ArasonElmanJacob87}, \cite{HwangJacob95}, \cite{Efrat99},
\cite[Ch.\ 26]{EfratBook}, \cite{Efrat07})
it was shown that there exist valuations satisfying (i) and (ii) if and only if $F$ has sufficiently many
$p$-rigid elements.
The dual notion in $G_{[2]}$ under the Kummer pairing can be interpreted, using certain Galois embedding problems,
in terms of the groups in $\calL_p$.

Connections between the group $G_{[3]}$ and valuations were earlier studied in \cite[\S7--8]{MaheMinacSmith04}
(for $p=2$) and also announced in \cite{Pop06b}.
This is also related to works by Bogomolov, Tschinkel, and Pop
(\cite{Bogomolov91}, \cite{Bogomolov92}, \cite{BogomolovTschinkel08}, \cite{Pop06a}),
showing that for function fields $F$ over algebraically closed fields, such ``tame" valuations can be recovered
from the larger Galois group $G/[[G,G],G]$.
For a nice survey with more references see \cite{BogomolovTschinkel10}.
Some other connections between rigidity and small Galois groups were previously
also investigated in \cite{AdemGaoKaragueuzianMinac01} and \cite{LeepSmith02}, and in connection
with absolute or maximal pro-$p$ Galois groups in, e.g.,
\cite{Efrat99}, \cite{Efrat00}, \cite{EnglerNogueira94}, and \cite{Koenigsmann03}.

Underlying our results is the fact, proved in \cite{EfratMinac2} (extending earlier
results in \cite{CheboluEfratMinac}), that for $G=G_F$ with $F$ as above, $G_{[3]}$ determines the
Galois cohomology ring $H^*(G,\dbZ/q)$, and is in fact the minimal Galois group of $F$ with this property.

For other works demonstrating the importance of the quotient $G_{[3]}$ in the Galois theory
of algebraic number fields see,  e.g.,  \cite{Koch02}, \cite{Morishita04}, \cite{Vogel05}.

\section{Galois-theoretic preliminaries}
We fix a prime number $p$.
For $p>2$ let
\[
H_{p^3}=\bigl\langle  r,s,t\ \bigm|\  r^p=s^p=t^p=[r,t]=[s,t]=1,\ [r,s]=t \bigr\rangle
\]
be the nonabelian group of order $p^3$ and exponent $p$ (the \textbf{Heisenberg group}).
Also let $D_4$ be the dihedral group of order $8$.
To make the discussion uniform, we set
\begin{equation}
\label{def of G bar}
\bar G=\begin{cases}
H_{p^3} & p>2 \\
D_4 & p=2.
\end{cases}
\end{equation}
In both cases, the Frattini subgroup of $\bar G$ is its center $Z(\bar G)$, and one has $\bar G/Z(\bar G)\isom(\dbZ/p)^2$.
Moreover, this is the unique quotient of $\bar G$ isomorphic to $(\dbZ/p)^2$.
Also, every proper subgroup of $\bar G$ is abelian.

From now on let $F$ be a field containing a fixed root of unity $\zeta_p$ of order $p$,
and let $G=G_F$ be its absolute Galois group.
The following theorem was proved in \cite[Thm. D]{EfratMinac2}.

\begin{thm}
\label{intersection for p odd}
Assume that $p>2$.
Then $G_{(3)}$ is the intersection of all normal open subgroups $N$ of $G$
such that $G/N$ is isomorphic to $\{1\}$, $\dbZ/p$ or $H_{p^3}$.
\end{thm}

The analog of this fact for $p=2$ was proved by Villegas \cite{Villegas88} and
Min\'a\v c--Spira \cite[Cor.\ 2.18]{MinacSpira96} (see also \cite[Cor.\ 11.3 and Prop.\ 3.2]{EfratMinac1}):

\begin{thm}
\label{intersection for p=2}
Assume that $p=2$.
Then $G_{(3)}$ is the intersection of all normal open subgroups $N$ of $G$
such that $G/N$ is isomorphic to $\{1\}$, $\dbZ/2$, $\dbZ/4$, or $D_4$.
\end{thm}
Moreover, $\dbZ/2$ can be omitted from this list unless $F$ is Euclidean \cite[Cor.\ 11.4]{EfratMinac1}.

Let $H^i(G)=H^i(G,\dbZ/p)$ be the $i$th profinite cohomology group with the trivial action of $G$ on $\dbZ/p$.
Thus $H^1(G)$ is the group of all continuous homomorphisms $G\to\dbZ/p$.
We write $\cup$ for the cup product $H^1(G)\times H^1(G)\to H^2(G)$.
For $a\in F^\times$ let $(a)_F\in H^1(G)$ correspond
to the coset $a(F^\times)^p$ under the Kummer isomorphism $F^\times/(F^\times)^p\xrightarrow{\sim} H^1(G)$.
One has $(a)_F\cup(a)_F=(a)_F\cup(-1)_F$ \cite[Prop.\ III.9.15(5)]{Berhuy10}.

Next for a finite group $K$, we call a Galois extension $E/F$ an \textbf{$K$-extension} if $\Gal(E/F)\isom K$.
We say that a $(\dbZ/p)\times(\dbZ/p)$-extension $F(\root p\of a,\root p\of b)/F$ embeds inside a $\bar G$-extension $E/F$
\textbf{properly} if either $p>2$ or else $p=2$ and $\Gal(E/F(\sqrt{ab}))\isom\dbZ/4$.

We refer to \cite[(6.1.8), (3.6.3), (3.6.2)]{Ledet05} for the following well known facts;
see also \cite{GrundmanSmithSwallow95} and \cite{GrundmanSmith96}.

\begin{lem}
\label{embedding problems}
Let $a,b\in F^\times$.
\begin{enumerate}
\item[(a)]
When $(a)_F,(b)_F$ are $\dbF_p$-linearly independent, $F(\root p\of a,\root p\of b)/F$
embeds inside a $\bar G$-extension properly if and only if $(a)_F\cup(b)_F=0$.
\item[(b)]
When $p=2$ and $(a)_F\neq0$, the extension $F(\sqrt a)/F$ embeds inside a $\dbZ/4$-extension
if and only if $(a)_F\cup(-1)_F=0$.
\end{enumerate}
\end{lem}

\section{Rigidity}
\label{section on rigidity}
The following key notion is a special case of \cite[Def.\ 23.3.1]{EfratBook} and originates from
\cite{Szymiczek77} and \cite{Ware81}.
Note however that our definition differs by sign from that of \cite{Ware81}.

\begin{defin} \rm
An element $a$ of $F^\times$ is called \textbf{$p$-rigid} if $(a)_F\neq0$ and there is no $b\in F^\times$ such that
$(a)_F\cup(b)_F=0$ in $H^2(G)$ and $(-a)_F,(b)_F$ are $\dbF_p$-linearly independent.
\end{defin}

To get an alternative description of $p$-rigid elements, we define subsets $C,D$ of $F^\times$ as follow.

When $(-1)_F=0$ (resp., $(-1)_F\neq0$) let $C$ be the set of all $a\in F^\times$ for which there exists $b\in F^\times$
such that $(a)_F\cup(b)_F=0$ and $(a)_F,(b)_F$ (resp., $(a)_F,(b)_F,(-1)_F$) are $\dbF_p$-linearly independent in $H^1(G)$.

When $(-1)_F\neq0$ (so $p=2$) we  set
\[
D=\{a\in F^\times\ |\ (a)_F\cup(-1)_F=0\}.
\]
It is the subgroup of $F^\times$.

\begin{lem}
\label{rigidity and Galois groups}
Let $a\in F^\times$ such that $(a)_F\neq0,(-1)_F$.
The following conditions are equivalent:
\begin{enumerate}
\item[(a)]
$a$ is not $p$-rigid;
\item[(b)]
either $a\in C$ or both $(-1)_F\neq0$ and $a\in D$;
\end{enumerate}
\end{lem}
\begin{proof}
When $(-1)_F=0$ this is immediate.

Next assume that  $(a)_F\cup(-1)_F\neq0$.
Then $(-1)_F\neq0$, $p=2$ and $(a)_F\cup(a)_F\neq0$.
Thus, if $b\in F^\times$ satisfies $(a)_F\cup(b)_F=0$, then $(b)_F\neq (-1)_F,(a)_F$.
Therefore $(-a)_F,(b)_F$ are $\dbF_2$-linearly independent if and only if
$(a)_F,(b)_F,(-1)_F$ are $\dbF_2$-linearly independent.
Conclude that in this case $a$ is not $2$-rigid if and only if $a\in C$.

Finally, assume that $(-1)_F\neq0$ but $(a)_F\cup(-1)_F=0$ (i.e., $a\in D$).
Then $p=2$ and, by the assumptions, $(-a)_F,(-1)_F$ are $\dbF_2$-linearly independent.
Hence $a$ is not $2$-rigid.
\end{proof}

Next let $N_p$ be the subgroup of $F^\times$ generated by all elements which are not $p$-rigid and by $-1$.

We will need the following result of Berman and Cordes which simplifies the definition in the case $p=2$
(see \cite[Example 2.5(i)]{Ware81}, \cite[Ch.\ 5, Th.\ 5.18]{Marshall80},
and the related result \cite[Th.\ 1]{BermanCordesWare80}):

\begin{prop}
\label{Berman}
Let $p=2$.
Then $N_2$ is the set of all $a\in F^\times$ such that $a$ or $-a$ is not $2$-rigid.
\end{prop}

\begin{cor}
\label{three possibilities}
One of the following holds:
\begin{enumerate}
\item[(1)]
$N_p=\langle (F^\times)^p,C,-1\rangle$;
\item[(2)]
$p=2$, $(-1)_F\neq0$ and  $N_2=\langle D,-1\rangle$.
\end{enumerate}
\end{cor}
\begin{proof}
If $(-1)_F=0$, then (1) holds by Lemma \ref{rigidity and Galois groups}.

Next suppose that $(-1)_F\neq0$ (so $p=2$).
By Lemma \ref{rigidity and Galois groups}, $a\in F^\times\setminus((F^\times)^2\cup-(F^\times)^2)$ is not $2$-rigid
if and only if it is in $C\cup D$.
Hence the subgroups $\langle (F^\times)^2,C,-1\rangle$  and $\langle D,-1\rangle$ of $F^\times$ are contained in $N_2$.
Conversely, by Proposition \ref{Berman}, $N_2$ is contained in the union of these two subgroups.
Thus
\[
N_2=\langle (F^\times)^2,C,-1\rangle\cup \langle D,-1\rangle.
\]
Since a group cannot be the union of two proper subgroups, (1) or (2) must hold.
\end{proof}

\begin{rem}
\rm
Let $K^M_*(F)$ be the Milnor $K$-ring of $F$ (\cite{Milnor70}, \cite[\S24]{EfratBook}).
The Kummer isomorphism $F^\times/(F^\times)^p\xrightarrow{\sim} H^1(G_F)$, $a(F^\times)_F\mapsto (a)_F$,
induces the Galois symbol homomorphism
$K^M_*(F)/pK^M_*(F)\to H^*(G_F)$.
By the Merkurjev--Suslin theorem (\cite{MerkurjevSuslin82}, \cite[Ch.\ 8]{GilleSzamuely}), it is an isomorphism in degree $2$
(in fact, by the more recent results of Rost and Voevodsky \cite{Voevodsky11},
it is an isomorphism in all degrees, but we shall not need this very deep fact).
Therefore, our notion of a $p$-rigid element $a$ coincides with the notion of $(F^\times)^p$-rigidity of $a(F^\times)^p$
in $K^M_*(F)/pK^M_*(F)$ given in \cite[Def.\ 23.3.1]{EfratBook}.
Consequently $N_p$ coincides with the subgroup $N_{(F^\times)^p}$, defined $K$-theoretically
in \cite[Def.\ 26.4.5]{EfratBook}.
\end{rem}

\section{The Kummer pairing}
Let $\mu_p$ be the group of $p$th roots of unity in $F$ and recall that $G=G_F$ is the absolute Galois group of $F$.
Consider the Kummer pairing
\[
(\cdot,\cdot)\colon\quad G\times F^\times\to\mu_p, \quad
(\sig,a)\mapsto \sig(\root p\of a)/\root p\of a.
\]
Its left kernel is $G^{(2)}$ and its right kernel is $(F^\times)^p$.
We compute the annihilator of $N_p$ under this pairing.

Let $T=\bigcap_\rho\rho\inv(2\dbZ/4\dbZ)$, where $\rho$ ranges over all epimorphisms $\rho\colon G\to\dbZ/4$,
and $2\dbZ/4\dbZ$ is the subgroup of $\dbZ/4$ of order $2$.

\begin{lem}
\label{annihilator of D}
Assume that $(-1)_F\neq0$.
The annihilator of $D$ with respect to the Kummer pairing is $T$.
\end{lem}
\begin{proof}
Let $\sig\in G$.
Then $\sig\in T$ if and only if $\sig$ fixes $\sqrt a$ for every $a\in F^\times\setminus(F^\times)^2$
such that $F(\sqrt a)/F$
embeds inside a $\dbZ/4$-extension of $F$.
By Lemma \ref{embedding problems}(b), this means that $(\sig,a)=1$ whenever $(a)_F\cup(-1)_F=0$,
i.e., whenever $a\in D$.
\end{proof}

We define a subgroup $\widetilde G$ of $G$ by $\widetilde G=G$ if $p>2$, and
$\widetilde G=G_{F(\sqrt{-1})}$ when $p=2$.
Thus $\widetilde G=G$ when $(-1)_F=0$.
Also let $\bar G$ be as in (\ref{def of G bar}).

\begin{prop}
\label{annihilator of C}
The following conditions on $\sig\in \widetilde G$ are equivalent:
\begin{enumerate}
\item[(a)]
for every $\tau\in\widetilde G$ the commutator $[\sig,\tau]$ is in $G_{(3)}$;
\item[(b)]
for every $\tau\in\widetilde G$ and every $\bar G$-extension $L$ of $F$, the restrictions $\sig|_L,\tau|_L$ commute;
\item[(c)]
$(\sig,a)=1$ for every $a\in C$.
\end{enumerate}
\end{prop}
\begin{proof}
(a)$\Leftrightarrow$(b): \quad
Let $\tau\in\widetilde G$.
By Theorem \ref{intersection for p odd} and Theorem \ref{intersection for p=2},
$[\sig,\tau]\in G_{(3)}$ if and only if $\sig|_L,\tau|_L$ commute in $\Gal(L/F)$ for every Galois extension $L/F$ with
Galois group in $\{1,\dbZ/p,H_{p^3}\}$, when $p>2$, or in $\{1,\dbZ/2,\dbZ/4,D_4\}$, when $p=2$.
When $\Gal(L/F)$ is abelian, the commutativity is trivial, so it is enough to consider $\bar G$-extensions $L/F$.

\medskip

(b)$\Rightarrow$(c):\quad
Let $a\in C$ and take $b\in F^\times$ as in the definition of $C$.
Then when $(-1)_F=0$ (resp., $(-1)_F\neq0$) the Kummer elements $(a)_F,(b)_F$ (resp., $(a)_F,(b)_F,(-1)_F$)
are $\dbF_p$-linearly independent.
Hence there exists $\tau\in \widetilde G$
such that $\tau(\root p\of a)=\root p\of a$ and $\tau(\root p\of b)\neq\root p\of b$.
Moreover, $(a)_F\cup(b)_F=0$, so Lemma \ref{embedding problems}(a) yields a $\bar G$-extension $L/F$
with $F(\root p\of a,\root p\of b)\subseteq L$.
By assumption, the restrictions $\sig|_L,\tau|_L$ commute.
But $\bar G$ is non-commutative, so these restrictions belong to a proper subgroup of $\Gal(L/F)$.
By the Frattini argument, their restrictions $\sig_1,\tau_1$ to
$\Gal(F(\root p\of a,\root p\of b)/F)\isom(\dbZ/p)^2$
belong to a proper subgroup, which is necessarily cyclic of order $p$.
Thus $\sig_1\in\langle\tau_1\rangle$, whence $\sig(\root p\of a)=\sig_1(\root p\of a)=\root p\of a$, as desired.

\medskip

(c)$\Rightarrow$(b):\quad
Let $\tau\in\widetilde G$ and let $L$ be a $\bar G$-extension of $F$.
Take $a,b\in F^\times$ such that $L_0=F(\root p\of a,\root p\of b)$ is a $(\dbZ/p)^2$-extension of $F$
which embeds properly in $L$.
By Lemma \ref{embedding problems}(a),  $(a)_F\cup(b)_F=0$.
In view of the structure of $\bar G$, the center $Z(\Gal(L/F))$ is $\Gal(L/L_0)$.

\textbf{Case 1:} \textsl{$\sig|_{L_0},\tau|_{L_0}$ do not generate $\Gal(L_0/F)$.}\
Then $\sig|_L,\tau|_L$ generate a proper subgroup of $\Gal(L/F)\isom \bar G$, which is necessary commutative.
Thus $\sig|_L,\tau|_L$ commute, as required.

\textbf{Case 2:} \textsl{$a,b\in C$.}\
By assumption, $(\sig,a)=(\sig,b)=1$.
Therefore $\sig|_L\in\Gal(L/L_0)=Z(\Gal(L/F))$, and we are done again.

\textbf{Case 3:} \textsl{$\sig|_{L_0},\tau|_{L_0}$ generate $\Gal(L_0/F)$ and at least one of $a,b$ is not in $C$.} \
By construction, $(a)_F,(b)_F$ are $\dbF_2$-linearly independent.
Hence necessarily $(-1)_F\neq0$, $p=2$, and $(a)_F,(b)_F,(-1)_F$ are $\dbF_2$-linearly dependent.
Without loss of generality, $(a)_F\neq(-1)_F$.
We obtain that
\[
\Gal(L/F(\sqrt a,\sqrt{-1}))=\Gal(L/L_0)=Z(\Gal(L/F)).
\]
If $\sig(\sqrt a)=\sqrt a$, then (as $\sig\in \widetilde G$),
\[
\sig|_L\in\Gal(L/F(\sqrt a,\sqrt{-1}))=Z(\Gal(L/F)).
\]
Similarly, if $\tau(\sqrt a)=\sqrt a$, then $\tau|_L\in Z(\Gal(L/F))$, and in both cases we are done.

Finally, if $\sig(\sqrt a)=\tau(\sqrt a)=-\sqrt a$, then $\sig,\tau$ coincide on $F(\sqrt a,\sqrt{-1})=L_0$.
Hence $\sig|_L,\tau|_L$  generate a proper subgroup of $\Gal(L/F)\isom\bar G$, which is necessarily abelian.
Therefore they commute.
\end{proof}

\begin{cor}
\label{annihilator of Np}
The annihilator of $N_p$ in $G$ with respect to the Kummer pairing is
\[
\widetilde Z=\begin{cases}
T\cap\widetilde G, & \mathrm{if\ } (-1)_F\neq 0 \mathrm{\ and\ } N_2=\langle D,-1\rangle \\
\{\sig\in \widetilde G\  |\ \forall\tau\in\widetilde G:\ [\sig,\tau]\in G_{(3)}\}, &  \rm otherwise.
\end{cases}
\]
\end{cor}
\begin{proof}
First we note that, since $-1\in N_p$, the annihilator of $N_p$ is contained in $\widetilde G$.
Now in case (1) (resp., (2)) of Corollary \ref{three possibilities},
the assertion follows from Proposition \ref{annihilator of C} (resp., Lemma \ref{annihilator of D}).
\end{proof}

Next let $\bar Z$ be the image of $\widetilde Z$ under the natural projection $G\to G^{[2]}=G/G^{(2)}$.
Then $\bar Z\isom\widetilde Z/(G^{(2)}\cap\widetilde Z)$.
Note that if $(-1)_F=0$, then $\bar Z$ is just the image of $Z(G_{[3]})$ in $G^{[2]}$.

\begin{cor}
\label{duality for Np}
The Kummer pairing induces a perfect pairing
\[
\bar Z\times(F^\times/N_p)\to\mu_p.
\]
\end{cor}
\begin{proof}
The Kummer pairing induces a perfect pairing
\[
G^{[2]}\times(F^\times/(F^\times)^p)\to\mu_p.
\]
By Corollary \ref{annihilator of Np}, the annihilator of $N_p/(F^\times)^p$ is $\bar Z$.
The assertion now follows from general Pontrjagin duality theory.
\end{proof}

\section{Rigid fields}
The field $F$ is called \textbf{$p$-rigid} if all $a\in F^\times\setminus(F^\times)^p$ are $p$-rigid.
The next result applies Corollary \ref{duality for Np} to characterize these fields in terms of $G_{[3]}$.
For $p>2$ the equivalence (a)$\Leftrightarrow$(e) was proved in \cite[Th.\ 14]{MassyNguyenQuangDo77}; see also
\cite{Ware92}.
For $p=2$ the equivalences (a)$\Leftrightarrow$(c)$\Leftrightarrow$(d) where earlier proved in
\cite[Th.\ 3.13]{MinacSpira90}.

\begin{thm}
Assume that $(-1)_F=0$.
The following conditions are equivalent:
\begin{enumerate}
\item[(a)]
$F$ is $p$-rigid;
\item[(b)]
$N_p=(F^\times)^p$;
\item[(c)]
$G_{[3]}$ is abelian;
\item[(d)]
When $p>2$ (resp., $p=2$), $G_{[3]}\isom(\dbZ/p)^I$ (resp., $G_{[3]}\isom(\dbZ/4)^I$) for some index set $I$;
\item[(e)]
$F$ has no $\bar G$-extensions.
\end{enumerate}
\end{thm}
\begin{proof}
(a)$\Leftrightarrow$(b):  \quad Immediate.

\medskip

(b)$\Leftrightarrow$(c):\quad
By Corollary \ref{duality for Np}, $N_p=(F^\times)^p$ if and only if $\bar Z\isom G^{[2]}$,
i.e., the natural map $G_{[3]}\to G^{[2]}$ maps $Z(G_{[3]})$ surjectively.
By the Frattini argument, this means that $G_{[3]}=Z(G_{[3]})$.

\medskip

(c)$\Rightarrow$(d): \quad
When $p>2$, we use that abelian profinite groups of exponent dividing $p$ always have the form $(\dbZ/p)^I$.
Similarly, when $p=2$ the group $G_{[3]}$ has exponent dividing $4$, and by assumption is abelian.
Moreover, since $(-1)_F=0$, every $\dbZ/2$-extension embeds in
a $\dbZ/4$-extension (Lemma \ref{embedding problems}(b)).
Hence $G_{[3]}$ has the form $(\dbZ/4)^I$.

\medskip

(d)$\Rightarrow$(c)$\Rightarrow$(e):  \quad
Immediate.

\medskip

(e)$\Rightarrow$(c): \quad
Use Theorem  \ref{intersection for p odd} (when $p>2$) and Theorem \ref{intersection for p=2} (when $p=2$).
\end{proof}

\begin{rem}
\rm
An analogous result was proved in \cite[Prop. 12.1 and Prop.\ 3.2]{EfratMinac1} for the larger quotient $G/G^{(3)}$,
where $G^{(3)}=(G^{(2)})^p[G^{(2)},G]$ is the third subgroup in the descending $p$-central filtrartion of $G=G_F$:
namely, when $p>2$ (resp., $p=2$) $G/G^{(3)}$ is abelian if and only if $F$ has no Galois extensions with Galois group
$M_{p^3}$  (resp., $D_4$), where $M_{p^3}$ denotes the unique nonabelian group of odd order $p^3$ and exponent $p^2$.
Note that indeed $G^{(3)}=G_{(3)}$ for $p=2$, by \cite[Remark 2.1(1)]{EfratMinac2}.
\end{rem}

\section{Valuations}
\label{section on valuations}
Throughout this section we assume that $(-1)_F=0$.
As we mentioned earlier, existence of $(F^\times)^p$-compatible valuations $v$ with $v(F^\times)\neq pv(F^\times)$
is related to $p$-rigid elements, and therefore to the group $N_p$.
Further, $F^\times/N_p$ is dual to the image $\bar Z$ of $Z(G^{[3]})$ in $G^{[2]}$ (Corollary \ref{duality for Np}).
Thus we can now detect these valuations from our Galois group $G^{[3]}$ under some finiteness conditions discussed below.

Denote the exterior (graded) algebra of an $R$-module $M$ by $\bigwedge^*_RM$.
There is a canonical graded ring epimorphism
$\bigwedge^*_{\dbF_p}(F^\times/(F^\times)^p)\to K^M_*(F)/pK^M_*(F)$.
We say that $(F^\times)^p$ is \textbf{totally rigid} if this map is an isomorphism
(see \cite[\S26.3 and Example 23.2.4]{EfratBook}).

\begin{exa}
\rm
Suppose that $F$ is equipped with an $(F^\times)^p$-compatible valuation $v$ such that $\bar F_v^\times=(\bar F_v^\times)^p$
and such that the induced map $F^\times/(F^\times)^p\xrightarrow{\sim}v(F^\times)/pv(F^\times)$ is an isomorphism.
E.g., this holds for $F=\dbC((t_1))\cdots((t_n))$.
In the terminology of \cite[\S23.2]{EfratBook}, let $\mathbf{0}[v(F^\times)/pv(F^\times)]$ be the extension of the trivial
$\kappa$-structure $\bf0$ by the abelian group $v(F^\times)/pv(F^\times)$.
One has
\[
\mathbf{0}[v(F^\times)/pv(F^\times)]=\textstyle\bigwedge^*_{\dbF_p}(v(F^\times)/pv(F^\times))
\]
 as graded rings \cite[Example 23.2.4]{EfratBook}.
Further, there is a natural isomorphism
\[
K^M_*(F)/pK^M_*(F)\xrightarrow{\sim}\mathbf{0}[v(F^\times)/pv(F^\times)]
\]
\cite[Th.\ 26.1.2 and Example 26.1.1(2)]{EfratBook}.
Hence $(F^\times)^p$ is totally rigid.
See  also \cite[\S26.8]{EfratBook}.
\end{exa}

For a valuation $v$ on $F$ let $\bar F_v$ be its residue field, and let $O_v^\times$ be its group of $v$-units.
We will need the following special cases of \cite[Prop.\ 26.5.1, Th.\ 26.5.5(c), Th.\ 26.6.1]{EfratBook}, respectively:

\begin{prop}
\label{thm from Efr06}
\begin{enumerate}
\item[(a)]
If $v$ is an $(F^\times)^p$-compatible valuation on $F$, then $N_p\leq (F^\times)^pO_v^\times$.
\item[(b)]
If $(F^\times)^p$ is not totally rigid, and either $p=2$ or $(F^\times:(F^\times)^p)<\infty$, then
there exists an $(F^\times)^p$-compatible valuation $v$ on $F$ with $N_p=(F^\times)^pO_v^\times$.
\item[(c)]
If $(F^\times)^p$ is totally rigid, then there exists an $(F^\times)^p$-compatible valuation $v$ on $F$ with
$((F^\times)^pO_v^\times:N_p)|p$.
\end{enumerate}
\end{prop}

We obtain our main result:

\begin{thm}
\label{existence of valuations}
\begin{enumerate}
\item[(a)]
If $v$ is an $(F^\times)^p$-compatible valuation on $F$, then $v(F^\times)/pv(F^\times)$ is a quotient of
the Pontrjagin dual $\bar Z^\vee$.
\item[(b)]
Assume that $(F^\times)^p$ is not totally rigid, and  that $p=2$ or $(F^\times:(F^\times)^p)<\infty$.
Then there exists an $(F^\times)^p$-compatible valuation $v$ on $F$ with
$v(F^\times)/pv(F^\times)\isom\bar Z^\vee$.
\item[(c)]
Assume that $(F^\times)^p$ is totally rigid.
There exists an $(F^\times)^p$-compatible valuation $v$ on $F$ and an epimorphism
$\bar Z^\vee\to v(F^\times)/pv(F^\times)$ with kernel of order dividing $p$.
\end{enumerate}
\end{thm}
\begin{proof}
By Corollary \ref{duality for Np}, $\bar Z^\vee\isom F^\times/N_p$.
Every valuation $v$ on $F$ induces an isomorphism $F^\times/(F^\times)^pO_v^\times\isom v(F^\times)/pv(F^\times)$.
Now use Proposition \ref{thm from Efr06}.
\end{proof}

From parts (a) and (b) we deduce:

\begin{cor}
\label{valuations and the dual}
Assume that $(F^\times)^p$ is not totally rigid, and  that $p=2$ or $(F^\times:(F^\times)^p)<\infty$.
Then there exists an $(F^\times)^p$-compatible valuation $v$ on $F$ with
$v(F^\times)\neq pv(F^\times)$ if and only if $\bar Z\neq \{1\}$.
\end{cor}

\begin{rem}
\rm
The finiteness assumption for $p\neq2$ in Proposition \ref{thm from Efr06}(b)
(and therefore in Theorem \ref{existence of valuations} and Corollary \ref{valuations and the dual})
originates from the chain argument in the proof of \cite[Prop.\ 26.5.4]{EfratBook}.
It is currently not known whether this assumption in Proposition \ref{thm from Efr06}(b) is actually necessary.
\end{rem}

\subsection*{Acknowledgements}
The first author was supported by the Israel Science Foundation (grant No.\ 23/09).
The second author was supported in part by National Sciences and Engineering Council of Canada grant R0370A01.
We also acknowledge the referee's suggestions which we use in our exposition.

\end{document}